\documentclass[
10pt]{amsart}
\usepackage{graphicx}
\usepackage{psfrag}

\newtheorem{theorem}{Theorem}[section]
\newtheorem{lemma}[theorem]{Lemma}

\newtheorem*{conjecture*}{Conjecture}

\newtheorem{proposition}[theorem]{Proposition}
\theoremstyle{definition}

\newtheorem{remark}[theorem]{Remark}
\theoremstyle{question}

\theoremstyle{acknowledgement}

\newtheorem*{theorem*}{Theorem}
\numberwithin{equation}{section}

\newcommand{\Er}{\mathbb{R}}

\newcommand{\calc}{{\mathcal{C}}}

\newcommand{\cale}{{\mathcal{E}}}

\newcommand{\calh}{{\mathcal{H}}}

\newcommand{\caln}{{\mathcal{N}}}

\newcommand{\be}{\mathbf{e}}
\newcommand{\bE}{\mathbf{E}}

\def\pmb#1{\setbox0=\hbox{#1}%
\kern-.02em\copy0\kern-\wd0
\kern0.01em\raise.01732em\copy0\kern-\wd0
\kern0.01em\raise.01732em\copy0\kern-\wd0
\lower.01732em\copy0\kern-\wd0 \kern0.01em\copy0\kern-\wd0
\kern0.01em\lower.01732em\box0}

\begin{document}

\title[Even Symmetry of Entire Solutions]{ Even Symmetry of Some Entire Solutions to the Allen-Cahn Equation in Two Dimensions }

\author{Changfeng Gui}
\address{Changfeng Gui,  Department of Mathematics, U-9, University of Connecticut\\
Storrs, CT 06269, USA  and School of Mathematics and Econometrics,
Hunan University, Changsha, China} \email{gui@math.uconn.edu}

\maketitle


\begin{abstract}
In this paper, we prove  even symmetry  and monotonicity of  certain
solutions of Allen-Cahn equation in a half plane. We also show that
entire solutions with {\it finite Morse index} and {\it four ends}
must be evenly symmetric with respect to  two orthogonal axes. A
classification scheme of general entire solutions with {\it finite
Morse index} is also presented  using energy quantization.
\end{abstract}



\section{ Introduction}

We shall consider entire solutions of the  following Allen-Cahn
equation
\begin{equation}\label{allencahn-general}
 u_{xx}+ u_{yy}-F'(u)=0, \quad |u| \le 1, \quad (x, y) \in \Er^2,
\end{equation}
where  $F$ is a  balanced double-well potential, i.e., $F \in C^{2
,\beta}([-1, 1])$  satisfies  $ F(1)=F(-1)=0$ and
\begin{equation}\label{doublewell}
\left\{
\begin{split}
&  F'(-1)=F'(1)=0, \quad  F''(-1)>0,\,\,  F''(1)>0;  \\
&F'(t)>0, \, t\in (-1, t_0); \quad   F'(t)<0, \,  t \in (t_0, 1)
\end{split}
\right.
\end{equation}
for some $t_0 \in (0, 1)$.  Without loss of generality, we may
assume that $t_0=0$.  A typical example of balanced double well
potential is  $  F(u)=\frac{1}{4} (1-u^2)^2, \quad u \in \Er $.

It is well-known that there exists a unique transition layer
solution $g(y)$ (up to translation) to the one dimensional
Allen-Cahn equation
\begin{equation}\label{1d-ac}
\left \{
\begin{split}
&g''(s)-F'(g(s))=0, \quad s \in \Er,\\
& \lim_{s \to \infty} g(s)=1,\quad \lim_{s \to -\infty} g(s)=-1.
\end{split}
\right.
\end{equation}
 We may assume that $g(0)=0$. Indeed,  $g$ is a minimizer of the following energy functional
$$
\bE(v):=\int_{-\infty}^{\infty} [\frac{1}{2} |v'|^2 + F(v)]dx
$$
in $\calh:=\{ v \in H^1_{loc}(\Er): -1 \le v \le 1, \,\,\lim_{ s \pm
\infty} v(s)=\pm 1\}$ and
$$
\be:=\bE(g)=\int_{-1}^{1} \sqrt{ 2 F(u)} du<\infty.
$$
The solution $g$ is non-degenerate in the sense that the linearized
operator has a kernel spanned only by $g'$.

If $u$ is an evenly symmetric solution in $x$, we may regard $u$ as
a solution in the half plane $\Er^2_+:= \{ (x, y) | x \ge 0, \,\, y
\in \Er\}$,
\begin{equation}\label{2d-ac}
\left\{
\begin{split}
 &u_{xx}+u_{yy}-F'(u)=0, \quad |u| <1,  \quad (x, y)
\in \Er^2_+\\
& u_x(0, y)=0, \quad y \in \Er.
\end{split}
\right.
\end{equation}

We may  also assume  that $u$ satisfies the  monotone condition
\begin{equation}\label{axial-monotone}
u_x(x, y) >0, \quad x>0, \,\, y \in \Er.
\end{equation}

Our main theorem  states that  $u$ must be evenly symmetric  with
respect to $y$ and  monotone for $ y>0$.

\begin{theorem}\label{axial-symm}
Assume that $u(x, y)$ is even in $x$ and satisfies \eqref{2d-ac} and
\eqref{axial-monotone}. Then $u$ is even in $y$, i.e.,
\begin{equation}\label{axial-even}
u(x, y)=u(x, -y), \quad  (x, y) \in \Er^2_+
\end{equation}
after a proper translation in $y$.  Moreover,
 $u_y(x,y)<0$ for $ x>0, y>0$, and  the 0-level set of $u$ for $ x>x_0$ can be expressed as
the graph of two $C^{3, \beta} $ functions  $ y=\pm k(x)$ which is
asymptotically linear, i.e. $k(x)= \kappa x +C+o(1)$  for some
constants  $\kappa>0, C$,  as $ x $ goes to infinity.

In particular, we have
\begin{equation}\label{limit}
\lim_{x \to \infty} u(x, y)=1, \quad \forall y \in \Er.
\end{equation}
\end{theorem}

This symmetry result may be regarded as the counterpart of De Giorgi
conjecture for a half plane.

We shall prove Theorem \ref{axial-symm}  in  in three main steps.
First, we carry out a preliminary analysis of the 0-level set
$\Gamma$  of $u$ and show  that  $\Gamma$  can be regarded as graphs
of two $C^{3, \beta}$ functions $y=k_i(x), i=1, 2$ for $x>x_0$ large
enough; Second,  we show that  $ k(x)$ must be asymptotically
linear. Finally we use the moving plane method to conclude.

We shall also discuss  the even symmetry of entire  solutions whose
asymptotically behavior at infinity are roughly prescribed. For
example, we can show that  an entire solution with {\it finite Morse
index} and  {\it four ends} must be evenly symmetric in both $x$ and
$y$, after  a proper translation and rotation.  For a finite integer
$m\ge 0$, we say that a solution $u$ defined in $\Omega \subset
\Er^n$ has {\it finite Morse index} $ m $ if $m$ is the maximal
dimension of any linear subspace of Sobolev space $H^1(\Omega)$
contained in
\begin{equation}\label{morse}
 \caln:=\{ \phi \in  H^{1}(\Omega) :  \int_{\Omega} |\nabla \phi|^2
 +F''(u)\phi^2 dV < 0 \} \cap \{0\}.
 \end{equation}
If $m=0$, $u$ is also called a {\it stable } solution  in $\Omega$.
If an entire solution $u$  has
 { \it finite Morse index}, then  $u$ must  be
  {\it stable }  outside a  large enough ball $B_{R_0}$  (see \cite{cabre3} and \cite{cabre4}).

An entire solution $u$ is called a solution with {\it $2k$  ends}
for some positive integer $k$ if the $0$-level set $\Gamma$ of $u$
outside a large disc $B_{R}(0)$ consists of $2k$  imbedded $ C^{1} $
curves $\Gamma_i:=\{ \bigl(r_i(t), \theta_i(t) \bigr): \forall t \ge
0 \} \,\, 1\le i\le 2k$  in polar coordinates, and  $ r_i(t) \to
\infty$ as $ t \to \infty,$
$$
\Gamma_i \subset \{ (r, \theta): r\ge R, \,\, \theta_i^- < \theta
<\theta_i^+, \quad  1 \le i \le 2k \}
$$
where  $ 0\le  \theta_i^-<\theta_i^+ < \theta_{i+1}^- <
\theta_{i+1}^+< 2\pi,\,\,\,  1 \le i \le 2k-1$.

 We have the following symmetry result for entire solutions with {\it four ends}.

\begin{theorem}\label{four-ends}  Suppose that $u$ is an entire
 solution to \eqref{allencahn-general} with  {\it finite Morse index} and {\it four ends}.
  Assume also
\begin{equation}\label{tech-4}
0<\theta_i^+-\theta_i^-<\pi, \quad  1\le  i \le 4.
\end{equation}
Then,  after a proper translation and rotation,  $ u$ satisfies
\begin{equation}\label{symmetry}
u(x, y)=u(x, -y)=u(-x, y), \quad \forall  (x, y) \in \Er^2
\end{equation}
and
\begin{equation}\label{monotone-u}
u_{x}(x, y)>0,  \quad u_{y}(x,y)<0, \quad \forall   x>0, \, \, y>0
\end{equation}
and \eqref{limit} holds.   Moreover, there exists an angle $
\Theta=2\theta \in (0, \pi)$ such that the $0$-level set of $u$ in
the first quadrant is a graph of a $C^{3, \beta}$ function $y=k(x)$
for $ x>X_0$ large enough, and
$$
k(x) =x \tan \theta +o(1), \quad \text{as } x \to \infty.
$$
\end{theorem}

An entire solution $u$  with {\it four ends} may be called a {\it
saddle } solution.  The above theorem may be regarded as a form of
De Giorgi conjecture for saddle solutions.  The angle $\Theta$ may
be called the {\it contact angle} of $u$ (see \cite{gui1} for more
discussion).

For a given $\Theta \in (0, \pi)$, the uniqueness of  {  \it four
ends} entire solutions with {\it contact angle } $\Theta$ is still
unknown.  It is stated  in \cite{wei}   that the formal dimension of
the moduli space of entire solutions with  {\it $2k$ ends} is $2k$.
For $k=2$, it means that  formally there is  local uniqueness of
{\it saddle} solutions with a fixed {\it contact angle}, up to a
translation and rotation.  However, the global uniqueness is a very
different and more difficult question.

The condition \eqref{tech-4} is a technical condition and is
believed to be unnecessary.  However, we need it for the proof of an
energy bound in Lemma \ref{energy-monotone}  for  a functional
\begin{equation}\label{functional}
\cale_R(u):=\int_{B_R} \bigl( \frac{1}{2} |\nabla u|^2 +F(u) \bigr)
dxdy.
\end{equation}
If we assume the energy bound Lemma \ref{energy-monotone} directly
instead of \eqref{tech-4}, the conclusion of Theorem 1.2 still
holds. Indeed, we have the following general  energy quantization
result. Note that a different energy quantization phenomenon  has
been shown for Ginzburg-Landau equation (see \cite{brezis}).

\begin{theorem}\label{quant}
Assume that $u$ is an entire solution of \eqref{allencahn-general}
with {\it finite Morse index}. Then  there holds either
\begin{equation}\label{energy-infinity}
\lim_{R \to \infty} \cale_R(u)/R = \infty,
\end{equation}
or
\begin{equation}\label{energy-limit}
\lim_{R \to \infty} \cale_R(u)/R = 2k \be
\end{equation}  for some positive integer $k$.

In the latter case,  $u$ must be an entire solution with {\it $2k$
ends},  and  the $0$-level set of $u$ must be asymptotically
straight lines. Moreover, if we denote  the directions of these
lines by
 $\nu_i= \langle \cos \theta_i, \sin \theta_i\rangle, 1\le i \le
2k$, then
\begin{equation}\label{balance}
\sum_{i=1}^{2k}\nu_i=\langle 0,0 \rangle.
\end{equation}
\end{theorem}

It is suspected that the first case in Theorem 1.3 may not happen at
all. It would be interesting to show  that only \eqref{energy-limit}
holds  and  for a given configuration $\nu_i, 1\le i\le k$ there
exist only two  corresponding solutions with opposite signs after a
proper translation. All entire solutions with {\it finite Morse
index} could then be classified accordingly.

We note that the existence of  entire solutions with {\it finite
Morse index} and {\it $2k$ ends } has been shown in \cite{wei}. It
was also pointed out in \cite{wei} that there may not be any
symmetry for entire solutions with six or more ends.   Note also
that \eqref{balance} implies \eqref{tech-4} for $k=2$.

The paper is organized as follows. In Section 2, some preliminary
results for entire solutions of Allen-Cahn equation in all
dimensions shall be stated. In Section 3, we will prove Theorem
\ref{axial-symm}. In Section 4,  A simpler version of  Theorem
\ref{four-ends} shall be proven.  Theorem \ref{four-ends} and the
energy quantization property will be proven in Section 5.

\section{Some Basic Properties}

In this section we shall state  some useful properties of entire
solutions to the Allen-Cahn equation.

We first state   a gradient estimate \eqref{allencahn-general} for
all dimensions which was proven in \cite{mod}.

\begin{proposition}
Assume that $F(s)\ge 0, \forall s \in [-1,1]$. Suppose that $u$ is a
solution to \eqref{allencahn-general}.  Then
\begin{equation}\label{gradient}
|\nabla u|^2(x,y)  \le 2F\bigl(u(x,y)\bigr), \quad (x,y) \in \Er^n.
\end{equation}
\end{proposition}

It is also well-known that $u$ has  the following exponential decay
with respect to distance from the level set.

\begin{proposition}\label{exponential1}
Assume that $u$ is a solution to \eqref{allencahn-general}. Then
there exists constants $C$ and $\nu>0$ such that
\begin{equation}\label{exponential}
|u^2-1|+|\nabla u|+|\nabla^2 u| \le C e^{ -\nu d(x, y)}
\end{equation}
where $d(x, y)$ is the distance to the 0-level set $\Gamma$ of $u$.
\end{proposition}

This property can be proven by comparing $u$ with a solution $u_R>0$
of the Allen-Cahn equation in a ball $B_R$ centered at $(x, y)$ with
zero boundary condition, where $R=d(x, y)$. (See, e.g., \cite{gg}.)

The following monotonicity property of energy is shown in
\cite{mod2}.
\begin{proposition}\label{modica2}
Assume that $u$ is a solution to \eqref{allencahn-general}. Then
$\cale_R(u)/R$ is increasing in $R$.
\end{proposition}

\section{Even symmetry of solutions in a half plane}

 We now consider  an entire solution $u$ which is even in $x$.  Note
 that $u$ may  be regarded as a solution  of \eqref{2d-ac} in a half
plane.

 We  first study  the limit of $u(x, y)$
as $x$ goes to infinity.

 Define
$$
u^{\tau}(x, y):= u( \tau+x, y), \quad x \ge -\tau, \, \, \forall y
\in \Er.
$$
It is easy to see that $u^{\tau}(x, y)$ converges to  some function
$u^+(y)>-1$ in $C^{3}_{loc}(\Er^2)$ as $\tau$ goes to infinity,  and
$u^+(y)$ satisfies one dimensional  Allen-Cahn equation
\begin{equation}\label{allencahn-1d}
u_{yy}-F'(u)=0, \quad y \in \Er.
\end{equation}
Let
$$
\sigma^{\tau} (x, y)= \frac{ u^{\tau}_{x}(x, y)}{ u^{\tau}_{x}(0,
0)}>0, \quad \forall x \ge -\tau, \,\, y \in \Er.
$$
By the Harnack inequality and the gradient estimate for elliptic
equations, we know that $\sigma^{\tau} (x, y)$ converges to
$\sigma^*(x, y)>0$  in  $C^{2}_{loc}(\Er^2)$ as $\tau$ goes to
infinity, and $\sigma^*(x,y)$ satisfies the linearized equation of
Allen-Cahn equation
\begin{equation}\label{linear-ac}
\sigma_{xx}+\sigma_{yy}-F''(u^+(y))\sigma=0, \quad (x,y) \in \Er^2.
\end{equation}
Hence $u^+$ is a stable solution of \eqref{allencahn-1d}.  Then
there are three possibilities for $u^+$:

\begin{enumerate}

\item[(i)] $u^+ \equiv 1$;

\item[(ii)] $ u^+(y)=g(y-K)$ for some constant $K$;

\item[(iii)] $ u^+(y)=g(K-y)$ for some constant $K$.

\end{enumerate}

The goal is to show that only (i) holds.  To do so,  we shall  prove
several basic properties for $u$.    The first property is an energy
estimate of $u$ on a line.

\subsection{Energy estimate}

 We first show a simple but  important lemma regarding   the energy
 of $u$ on  $y$-axis.

\begin{lemma}
Suppose that $u$ is a solution to \eqref{2d-ac} and
\eqref{axial-monotone}.  Then
\begin{equation}\label{int-ur}
\int_{\Er}  [ F \bigl(u(0, y)\bigr) +\frac{1}{2} u_y^2(0, y)]dy<
3\be.
\end{equation}
\end{lemma}

\begin{proof}

Define
$$
h(y)=\int_{0}^{\infty} u_y  u_x dx, \quad \forall y \in \Er.
$$
In view of \eqref{gradient} and  the positivity of $u_x$, it is easy
to see that $h(y)$ is well-defined and
$$
|h(y)|< \int_{0}^{\infty} \sqrt{2 F\bigl(u(x,y)\bigr) } \cdot u_xdx
\le \be-G\bigl(u(0, y)\bigr)<\be, \quad  \forall  y \in \Er
$$
where
\begin{equation}\label{G}
G(t):=\int_{-1}^{t} \sqrt{2 F(s)} ds, \quad \forall t \in [-1, 1].
\end{equation}

Differentiating $h(y)$ with respect to $y$ and using \eqref{2d-ac},
we obtain
\begin{equation}\label{estimate-ur}
\begin{split}
&h'(y)=\int_{0}^{\infty}( u_{yy}u_x+ u_{y} u_{xy}) dx\\
=&\int_{0}^{\infty}  [\frac{\partial}{\partial x}
\bigl(F(u)-\frac{1}{2} u_{x}^2
 +\frac{1}{2} u_y^2  \bigr) ]dx\\
=&  [ F \bigl(u^+( y)\bigr) +\frac{1}{2} (u^+_y)^2(y) ] -[ F
\bigl(u(0, y)\bigr) +\frac{1}{2} u_y^2(0, y) ].
\end{split}
\end{equation}
Here we have used the facts  $u_x(0, y)=0$ and $\lim_{x \to \infty}
u_x(x,y)=0, \forall y \in \Er$.  Then,  we derive
\begin{equation}\label{int-ur2}
\begin{split}
 & \int_{a}^{b} [ F \bigl(u(0, y)\bigr)
+\frac{1}{2} u_y^2(0, y) ]dy  \\
=&\int_{a}^{b} [ F \bigl(u^+( y)\bigr) +\frac{1}{2} (u^+_y)^2( y) ]
dy  + \bigl(h(a)-h(b) \bigr).\\
\end{split}
\end{equation}

Define
\begin{equation}\label{rho}
\rho(x)= \int_{\Er} [ F \bigl(u(x,y)\bigr) +\frac{1}{2} u_y^2(x,
y)-\frac{1}{2} u_{x}^2 (x, y)]dy
\end{equation}
and
\begin{equation}\label{rho^+}
\rho^+= \int_{\Er} [ F \bigl(u^+(y)\bigr) +\frac{1}{2} (u^+_y)^2(
y)]dy.
\end{equation}

 Then, letting $a \to -\infty $ and $b \to +\infty$ in \eqref{int-ur2},    in view of the bound of
 $h(y)$ we obtain
\begin{equation}\label{int-ur3}
  \rho(0) = \rho^+ \ + \lim_{a \to -\infty} h(a)-\lim_{b \to \infty} h(b) \le 3\be.
\end{equation}

Therefore,  \eqref{int-ur} is proven.

\end{proof}

\subsection{ A Hamiltonian identity}

Next we  shall show a Hamiltonian identity for  solutions of
\eqref{2d-ac}.

\begin{lemma}\label{axial-hamiltonian}\
Assume that $u(x,y)$ satisfies \eqref{2d-ac} and
\eqref{axial-monotone}. Then
\begin{equation}\label{eq:axial-ham}
\rho(x)=\rho(0), \quad \forall x \in \Er_+.
\end{equation}
\end{lemma}

\begin{proof}

By  \eqref{int-ur} and the boundedness of $u$ in $C^3(\Er^n)$, we
know that  the following limits exist
$$
v^+:=\lim_{ y \to \infty} u(0, y), \quad  v^-:=\lim_{ y \to -\infty}
u(0, y)
$$
and
$$ |v^+|=1, \quad |v^-|=1.$$
Indeed, by the standard  translation argument it can be shown that
$$
v^+(x, y):=\lim_{ t \to \infty} u(x, y+t), \quad v^-(x, y):=\lim_{ t
\to -\infty} u(x, y+t)
$$
exist and are solutions to \eqref{2d-ac}, and hence
$$
v^+(x, y) \equiv v^+, \quad v^-(x, y) \equiv v^-, \quad (x,y) \in
\Er^2.
$$

In particular,
\begin{equation}\label{limit-ur}
\lim_{ |y| \to \infty} u_x(x, y)=0,  \quad \lim_{ |y| \to \infty}
u_y(x, y)=0,\quad \forall x \ge 0.
\end{equation}

Define
$$
h(R, y):=\int_{0}^{R} u_y  u_x dx, \quad \forall y \in \Er.
$$
Then, in view of \eqref{limit-ur}, we have
$$
\lim_{|y| \to \infty} h(R, y)=0, \quad \forall R \ge 0.
$$

As before, differentiating $h(R, y)$ with respect to $y$ and using
\eqref{2d-ac},  we can obtain
\begin{equation*}
\begin{split}
h'(R, y)=&\int_{0}^{R}( u_{yy}u_x+ u_{y} u_{xy}) dx\\
=&\int_{0}^{R}  [\frac{\partial}{\partial x} \bigl(F(u)-\frac{1}{2}
u_{x}^2  +\frac{1}{2} u_y^2  \bigr)]dx\\
=& [ F \bigl(u(R, y)\bigr) +\frac{1}{2} u_y^2(R, y) -\frac{1}{2}
u_x^2(R, y)]-[ F \bigl(u(0, y)\bigr) +\frac{1}{2} u_y^2(0, y) ].
\end{split}
\end{equation*}
Then, integrating the above with respect to $y$ in $\Er$,  we derive
\begin{equation}\label{int-ur-R}
 \rho(0)-\rho(R)=\lim_{a \to -\infty} h(R, a) -\lim_{b \to \infty} h(R, b)=0.
\end{equation}
The lemma is proven.

\end{proof}

We can indeed show the following limit.
\begin{lemma}\label{limit-u}
\begin{equation}\label{limit-u0}
\lim_{ |y| \to \infty} u(x, y)=-1, \quad \forall x \in \Er.
\end{equation}
\end{lemma}

\begin{proof}

We shall show the lemma by considering different cases.

In Case (i), i.e.,  $u^+ \equiv 1$,  there are four possibilities:

\noindent
\begin{enumerate}
\item   $ v^+ = 1, \quad v^- =1;$
\item   $ v^+ = -1, \quad v^- =1;$
\item   $ v^+ = 1, \quad v^- =-1;$
\item   $ v^+ = -1, \quad v^- =-1;$

\end{enumerate}

From \eqref{int-ur3} and the Hamiltonian identity
\eqref{eq:axial-ham} we have
\begin{equation}\label{limit-h1}
\lim_{a \to -\infty} h(a)-\lim_{b \to \infty} h(b)+\rho^+= \rho(0)=
\lim_{x \to \infty} \rho(x).
\end{equation}

In Subcase (1),  we can estimate
$$
\lim_{a \to -\infty} |h(a)|\le \lim_{a \to -\infty}[G(1)- G\bigl(
u(0,a)\bigr)] =0
$$
and
$$
\lim_{b \to \infty} |h(b)|\le \lim_{b \to \infty} [G(1)- G\bigl(
u(0,b)\bigr) ]=0.
$$
Then \eqref{int-ur3} becomes
\begin{equation*}
  \rho(0) =\lim_{a \to -\infty} h(a)-\lim_{b \to \infty} h(b)=0.
\end{equation*}
This is a contradiction, and therefore Subcase (1) is excluded.

In Subcase (2),  we can estimate
$$
\lim_{a \to -\infty} |h(a)|\le \lim_{a \to -\infty} [ G(1)-G\bigl(
u(0,a)\bigr)] \le \be
$$
and
$$
\lim_{b \to \infty} |h(b)|\le \lim_{b \to \infty}[G(1)- G\bigl(
u(0,b)\bigr)] =0.
$$
Then \eqref{int-ur3} becomes
\begin{equation*}
\rho(0)\le \be.
\end{equation*}

On the other hand, by the definition of $\be$, we have $ \rho(0) \ge
\be $.  Then we have $ u(0, y) = g(\pm y+ K_1)$ for some $K_1 \in
\Er$. Then $u(x, y)-u(0, y) $  is nonnegative  and satisfies a
linearized equation of \eqref{allencahn-general}. By the Harnack
inequality, we can derive $u(x, y) \equiv u(0, y)$.  This
contradicts with \eqref{axial-monotone},  and hence Subcase (2) is
excluded.  Subcase (3) is similar to Subcase (2). The lemma then
follows easily from \eqref{limit-ur}.

\noindent In Case (ii), i.e., $u^+(y)= g(y-K)$,  in view of the
monotone condition \eqref{axial-monotone} we know only Subcases (2)
and (4) are possible. If Subcase (2) happens, then \eqref{int-ur3}
becomes
\begin{equation*}
  \rho(0)=\be+\lim_{a \to -\infty} h(a)-\lim_{b \to \infty} h(b)=\be.
\end{equation*}
 Since $\rho(0) \ge \be $,  we get a contradiction immediately  as in Case (i).    Therefore
 Subcase (2) is excluded.

 Case (iii) is similar to Case (ii).   In all cases, we have proven
 that only Subcase (4) holds.  Hence  \eqref{limit-u0} is proven.

\end{proof}

In the level set analysis below,  we shall focus on  Case (i):  $u^+
\equiv 1$.   The other two cases can be discussed similarly with
minor modifications, and can be excluded eventually at the end of
this section.

In view of  \eqref{axial-monotone} and \eqref{limit-u0},  the
0-level set $\Gamma$ of $u$ can be represented by the graph of a
function $x=\gamma(y)$ which is defined for $ y\le K_1,$ and $ y\ge
K_2$ with $K_1\le K_2$ and is $C^3$.   By Lemma \eqref{limit-u}, we
also know
\begin{equation}\label{limit-gamma}
\lim_{ |y| \to \infty} \gamma(y) = \infty.
\end{equation}

\subsection{ The slope of the level set has a limit}
First we show the  limits of $\gamma'(y)$ exist  as $y \to \pm
\infty$.
\begin{lemma}\label{limit-gamma-der}
There exist  $\theta_1 \in [0, \pi/2]$ and $\theta_2  \in [-\pi/2,
0] $ such that
\begin{equation}\label{lim-slope}
\lim_{ y \to \infty} \gamma'(y)=\tan{\theta_1}, \quad
 \lim_{ y \to -\infty}
\gamma'(y)=\tan{\theta_2}.
\end{equation}
Here we use the convention that  $\tan{(\pi/2)}=\infty,
\tan{(-\pi/2)}=-\infty$.
\end{lemma}

\begin{proof}

For any sequence $\{y_m\}$  and constant $\theta \in [-\pi/2, \pi/2]
$ with  $|y_m| \to \infty$ and
$$
\lim_{m \to \infty} \gamma'(y_m)= \tan{\theta},
$$
we define
$$
u^m(x, y):= u(x+\gamma(y_m), y+y_m), \quad x  \ge -\gamma(y_m), \,
\, y \in \Er.
$$
Then $u^m$ converges to $u^*$ in $C^3_{loc} (\Er^2)$  after taking a
subsequence if necessary, where $u^*$ is a solution of
\eqref{allencahn-general} with $\frac{\partial {u^*}}{\partial x}(x,
y) \ge 0, \quad (x, y) \in \Er^2 $.  By the Harnack inequality, we
know that  either $ \frac{\partial {u^*}}{\partial x}(x, y) \equiv
0$, or
 $\frac{\partial {u^*}}{\partial x}(x, y) > 0, \quad (x, y) \in
\Er^2$.   In the first case, we define
$$
\sigma^m (x, y)= \frac{ u^m_{x}(x, y)}{ u^m_{x}(0, 0)}>0, \quad
\forall x \ge -\gamma(y_m), \,\, y \in \Er.
$$
By the Harnack inequality and the gradient estimate for elliptic
equations, we know that $\sigma^m (x, y)$ converges  along  a
subsequence to $\sigma^*(x, y)>0$  in  $C^{2}_{loc}(\Er^2)$   as $m$
goes to infinity. Furthermore,  $\sigma^*(x,y)$ satisfies the
linearized equation of Allen-Cahn equation at $u^*$
\begin{equation}\label{linear-ac2}
\sigma_{xx}+\sigma_{yy}-F''(u^*)\sigma=0, \quad (x,y) \in \Er^2.
\end{equation}

Hence $u^*$ is stable in both cases. By the De Giorgi conjecture for
$n=2$ (\cite{gg}), we know that $ u^*$ depends only on one
direction.   Since $u^*(0, 0)=0 $,  we conclude
\begin{equation}\label{lim-u}
u^*(x, y)=g\bigl(x \cos{\theta}-y\sin{\theta} \bigr), \quad \forall
(x, y) \in \Er^2.
\end{equation}
Note that  straightforward computations  can lead to
\begin{equation}\label{limit-rho}
\rho^*(\theta):=\int_{\Er} [ F \bigl(u^*(x, y)\bigr) +\frac{1}{2}
(u^*_y)^2(x, y)-\frac{1}{2} (u^*_{x})^2 (x, y)]dy= \be \sin{\theta}.
\end{equation}
(See, e.g., \cite{gui1}.)

Next we  shall show  the first limit in \eqref{lim-slope}.

Let
\begin{equation}\label{lim-sup}
\limsup_{ y \to \infty} \gamma'(y)=\tan{\theta_1}
\end{equation}
for some $\theta_1 \in [0, \pi/2]$.

 If
$\liminf_{ y \to \infty} \gamma'(y)=\tan{\theta_0}< \tan{\theta_1}$
for some $\theta_0 \in [-\pi/2, \theta_1)$,  then, for any fixed
$\theta \in (\theta_0, \theta_1)$  there exists a sequence $\{y_m\}$
with $\lim_{ m \to \infty} \gamma'(y_m)=\tan{\theta}$ and $y_m \to
\infty$ as $ m \to \infty$.

For any fixed $R>0$, by the monotone condition
\eqref{axial-monotone} and \eqref{lim-u} we have
\begin{equation}\label{estiamte-h}
\begin{split}
&\lim_{m \to \infty} h(y_m)  =  \lim_{m \to \infty} \int_{
\gamma(y_m)-R}^{\gamma(y_m)+R} u_x u_ydx  \\
& +  O(1) \cdot \lim_{m \to \infty} [G(1)-G \bigl( u(\gamma(y_m)+R,
y_m ) \bigr)] \\
&+ O(1) \cdot \lim_{m \to \infty} [G \bigl( u(\gamma(y_m)-R, y_m
\bigr)-G
\bigl( u(0, y_m) \bigr)]\\
&=-\sin{\theta} [G(\bigl( g(R \cos{\theta} )\bigr) -G(\bigl(
g(-R\cos{\theta} )\bigr)]+  O(1) [ G \bigl( g(-R) \bigr)]
\end{split}
\end{equation}
  where $G$ is defined in \eqref{G} and $ O(1)$ is with respect to  $R \to \infty$.  Letting
$R$ go to infinity, we obtain
$$
\lim_{m \to \infty} h(y_m)= -\be \sin{\theta}.
$$
By \eqref{int-ur2}, we know that $\lim_{a \to \infty} h(a)   $
exists and hence
$$
\lim_{y \to \infty} h(y)= -\be \sin{\theta}.
$$
This leads to
$$
\lim_{y \to \infty} \gamma'(y)= \tan{\theta}
$$
which contradicts \eqref{lim-sup}.  Therefore the first limit in
\eqref{lim-slope} is proven.

Similarly, we can show the second limit in \eqref{lim-slope}.

\end{proof}

Furthermore,   by \eqref{limit-h1}   we have
\begin{equation}\label{angle-relation}
\be( \sin{\theta_1}-\sin{\theta_2} ) =\rho(0) =\lim_{R\to \infty}
\rho( R )
\end{equation}

We note that in Case (ii), the above discussion can be modified with
$\theta_2=-\pi/2$ and
 $ y \to -\infty$  being  replaced by $ y \to K$.
Similar modifications can be done for Case (iii) with $\theta_1
=\pi/2$.

\subsection{The limits of slopes differ by a sign}

We shall  show that the limits of the slopes of the level set differ
only by a sign, i.e.,  $\theta_1=-\theta_2 \in (0, \pi/2)$.

\begin{lemma}\label{angle-lemma}
There holds
\begin{equation}\label{angle}
  \theta_1=-\theta_2.
\end{equation}
\end{lemma}

\begin{proof}

Recall that $u$ is an {\it even } solution in $\Er^2$ with respect
to $x$.

Let us choose an angle $\theta \in (0, \pi/2),  \theta \not =
\theta_1, -\theta_2$ and a Cartesian coordinate system $(z_1, z_2)$
such that $z_1$-axis and $y$ axis  form an angle $\theta$. In other
words, we have  $ x=z_1 \sin{\theta}+z_2 \cos{\theta},  y=z_1
\cos{\theta}-z_2 \sin{\theta}$.
 By \eqref{exponential}, we
know that
\begin{equation}\label{exponential2}
|u^2(z_1, z_2)-1|+|\nabla u(z_1, z_2)|+|\nabla^2 u(z_1, _2)| \le C
e^{ -\nu_1 |z_1|}, \quad \forall z_1 \in \Er
\end{equation}
for some positive constants $\nu_1>0$ and  $C$.

Therefore, there holds a Hamiltonian identity  like
\eqref{eq:axial-ham} with respect to $z$. Namely,
\begin{equation}\label{ham-z}
\bar{\rho}(\theta, z_2):= \int_{\Er} [ F \bigl(u(z_1, z_2)\bigr)
+\frac{1}{2} u_{z_1}^2(z_1, z_2)-\frac{1}{2} u_{z_2}^2 (z_1,
z_2)]dz_1 =\bar{\rho}(\theta,0)< \infty.
\end{equation}
(The proof is  similar to  \eqref{eq:axial-ham}; See also Theorem
1.1 in \cite{gui1}.)     When $ \theta>\theta_1, \theta>-\theta_2$,
a straightforward computation can lead to
\begin{equation}\label{ham-const}
\left\{
\begin{split} &\lim_{z_2 \to \infty} \bar{\rho}(\theta, z_2)=
\be
\bigl(\sin(\theta-\theta_1)+ \sin(\theta+ \theta_1) \bigr);\\
&\lim_{z_2 \to -\infty} \bar{\rho}(\theta, z_2)= \be
\bigl(\sin(\theta-\theta_2)+ \sin(\theta+ \theta_2) \bigr)
\end{split}
\right.
\end{equation}
Then we have
$$
\sin(\theta-\theta_1)+ \sin(\theta+ \theta_1)=\sin(\theta-\theta_2)+
\sin(\theta+ \theta_2)
$$
and hence  $\theta_1=-\theta_2$.

The same conclusion can be reached if  $\theta$ is in other range
compared to $\theta_1,  -\theta_2$, with only slight difference in
the expression in \eqref{ham-const}.  The details is left to the
reader. See also (2.12) in \cite{gui1}.

\end{proof}

Since $\rho(0)>0$,  an easy consequence of Lemma \ref{angle-lemma}
and \eqref{angle-relation} is $\theta_1=-\theta_2>0$, .  Next we
shall show $ \theta_1 < \pi/2$.

If $\theta_1=\pi/2$, we choose $\theta \in (0, \pi/2)$ and carry out
the same computation as \eqref{ham-const} to obtain
\begin{equation}\label{ham-const2}
\bar{\rho}(\theta, 0)=\lim_{z_2 \to \infty} \bar{\rho}(\theta,z_2)=2
\be \sin(\pi/2-\theta).
\end{equation}
Letting $\theta \to \pi/2$,  we obtain
$$
\lim_{\theta \to \pi/2} \bar{\rho}(\theta, 0)=0.
$$

On the other hand, by \eqref{gradient} we have
$$
\lim_{\theta \to \pi/2} \bar{\rho}(\theta, 0) \ge \int_{\Er}
\frac{1}{2} u_{x}^2(x,0)dx>0.
$$
This is a contradiction, and hence proves $ \theta_1 < \pi/2$.

\subsection{The level set is asymptotically a straight line}
Below we quote a lemma from \cite{gui1} on the asymptotical behavior
of the level set.

 \begin{lemma}\label{straightline}
Suppose that  $u(y_1, y_2)$ is a solution of \eqref{2d-ac} in a cone
$ \calc:=\{ y \in  \Er^2 :  |y_1 | \le  y_2 \tan{\alpha_0} , \,
y_2\ge M>0\}$ for some $0<\alpha_0<\pi/2$. The $0$-level set of $u$
in $\calc$ is given by the graph of a function $ y_1=k(y_2)$. Assume
\begin{equation}\label{zero-slope}
\lim_{y_2 \to \infty} k'(y_2) =0.
\end{equation}
 Then there is a finite number $A$ such that
\begin{equation}
\lim_{y_2 \to \infty} k(y_2) =A_1.\end{equation}
\end{lemma}

 The lemma can be shown in three steps.  First, we show that  an energy of  $u$ on
 a line segment $[-y_2\tan{\alpha}, y_2\tan{\alpha_0}], \quad \alpha \in (0, \alpha_0) $ is
 exponentially close to $\be$ as $y_2$ tends to $\infty$.  Second, we construct an
 optimal approximation of $u(\cdot, y_2)$  by  a shift of the  one dimensional solution
 $g \bigl(y_1-l(y_2) \bigr)$,
 and show that the error is exponentially small in $L^2$ norm as $y_2$
  goes to infinity. Finally, we deduce that the shift $l(y_2)$ has a finite limit,  and then
   conclude that
 $k(y_2)$ has a finite limit. For the details of the proof, the reader is referred to \cite{gui1}.

 Now we choose the coordinate system $ (y_1, y_2)$ so that
 $y_2$-axis  form an angle $\theta_1$  with $y$-axis, and
 $\alpha_0 <\min\{\pi/2-\theta_1, \theta_1\}$.  Using Lemma \ref{limit-gamma-der} and Lemma
 \ref{straightline}, we conclude that
 \begin{equation}\label{line+}
 \gamma(y)=(\tan{\theta_1}) y+
 A_2+o(1), \quad \text{ as }\,\, y \to \infty.
 \end{equation}
  Similarly,  we can show
  \begin{equation}\label{line-}
 \gamma(y)=-(\tan{\theta_1}) y+
 A_3+o(1), \quad \text{ as }\,\, y \to -\infty.
 \end{equation}
 Then,  for $Y_0$ large enough, the  inverse functions of $\gamma(y)$ for  $ y>Y_0$ and $y<Y_0$
 exist,  and may  be written as $y=k_1(x),  y=k_2(x) $ respectively.
  Moreover,
\begin{equation}\label{limit-cons}
k_1(x)=\kappa x+ B_1 +o(1), \quad k_2(x)=-\kappa x+ B_2 +o(1)
\end{equation}
as $ x \to \infty$, where $\kappa= \cot{\theta_1}$ is a positive
(finite) constant, and $B_1, B_2$ are constants.

\subsection{The moving plane method}
Next we shall use the moving plane method to show the even symmetry
of $u$ with respect to $y$.   Due to the fact that the asymptotical
behavior of $u$ is not homogeneous near infinity, in particular,
there is a transition layer along  the $0$-level set, the classic
moving plane method has to be carefully modified.   Indeed, we have
to use the exact asymptotical formulas of the 0-level sets $y
=k_i(x), i=1, 2$ near infinity as well the asymptotical behavior of
$u$ along these curves.

For this purpose, we define  $ u_\lambda(x, y):= u(x, 2\lambda-y)$
and $w_\lambda:=u_{\lambda} -u$ in $ D_{\lambda}:=\{ (x, y) : x \ge
0,\,\,  y \ge \lambda \} $.

\begin{lemma}
When $ \lambda $ is sufficiently large,   there holds $w_\lambda
>0$ in $ D_{\lambda}$.
\end{lemma}

\begin{proof}

We  first  fix $X_0$ sufficiently large so that $k_1(x), k_2(x) $
are well defined.  By the property of double well potential
\eqref{doublewell},  there exists a sufficiently small constant
$\delta>0$ such that  $ F''(t)>0, t \in [-1,
-1+\delta]\cup[1-\delta, 1]$.  There is also  a sufficiently large
constant $R_1>0$ such that $  -1 <g(s) \le -1+\delta/2, \, \forall s
<-R_1$ and $ 1-\delta/2 \le g(s) < 1, \,  \forall s >R_1$, where $g$
is the one dimensional solution in \eqref{1d-ac}.  By
\eqref{limit-cons} and \eqref{lim-u}, there exist $X_1, R_2$
sufficiently large such that for $ x>X_1$
\begin{equation}\label{asym-k}
\left\{
\begin{split}
&u(x, y) <-1+\delta, \quad \text{if }\,\, y> k_1(x)+R_2,\,  \text{ or } \, y <-k_2(x) -R_2,\\
&u(x, y) >1-\delta, \quad  \text{if }\,\, 0 < y<k_1(x)
-R_2,\, \text{ or }\,   -k_2(x) +R_2<y<0,\\
&|u(x, y)+ g\bigl( y\sin\theta_1 -x\cos\theta_1 -B_1
\sin\theta_1\bigr)|\le \delta/2,\\
&\quad \quad \quad  \quad \text{ if } k_1(x) -R_2 < y< k_1(x) +R_2;\\
&|u(x, y)-g\bigl( y\sin\theta_1+ x\cos\theta_1 -B_2
\sin\theta_1\bigr)|\le \delta/2,\\
&\quad \quad \quad  \quad \text{ if } k_2(x)-R_2 < y< k_2(x)+R_2.
\end{split}
\right.
\end{equation}

When $\lambda> \lambda_1$ is sufficiently large, by
\eqref{limit-cons} we have
$$
k^{\lambda}_2(x): = 2 \lambda -k_2(x) \ge k_1(y)+  R_2, \quad
\forall x \ge X_1.
$$
By Lemma \ref{limit-u},  we can also choose $\lambda_1$ so that
$$
u(x, y) <-1+\delta, \quad 0<x<X_1, \, y>\lambda_1.
$$

We claim that $w_\lambda \ge 0$ in $D_{\lambda}$ for $\lambda>
\lambda_1$, and shall show this claim in the  following three
subsets of $D_{\lambda}$ respectively:
\begin{equation*}
\begin{split}
&D_{\lambda}^+:= \{ (x, y):  0<x<X_1,\, y>\lambda, \text{or }
x>X_1,\, y> k^{\lambda}_2(x)\},\\
&D_{\lambda}^-:= \{ (x, y): x>X_1,\, y< k_1(x)\},\\
&D_{\lambda}^0:= \{ (x, y): x>X_1, \,  k_1(x)< y<k^{\lambda}_2(x)\}.
\end{split}
\end{equation*}

 If the claim  is not true in $D_{\lambda}^+$,  then there exists a sequence of points
$\{(x_m, y_m)\}_{m=1}^\infty \in  D_{\lambda}^+$ such that
$$
\lim_{m \to \infty} w_\lambda (x_m, y_m)=  \lim_{m \to \infty}
\bigl( u_\lambda(x_m, y_m)-u(x_m, y_m) \bigr) =\inf_{D_{\lambda}^+}
w_\lambda(x, y)<0.
$$
It can be seen   from \eqref{asym-k} that $u_{\lambda}(x_m, y_m)<
u(x_m, y_m)< -1+\delta$ when $m $ is large enough.  Then we can use
the standard translating arguments to obtain a contradiction as
follows. Define $w^m_\lambda(x, y):= w_\lambda( x+x_m, y+y_m)$ in
$D_{\lambda}^+-(x_m, y_m)$. Then $ w^{m}_\lambda $ converges to
$w^{\infty}_\lambda (x, y)$ in $C^3_{loc} (D^{\infty})$ for some
piecewise Lipschitz  domain $ D^{\infty}$ in $\Er^2$ which contains
a small ball centered at the origin. Furthermore,
$w^{\infty}_\lambda$ attains its negative minimum at the origin and
satisfies a linearized equation
\begin{equation}\label{elliptic-w}
w_{xx}+ w_{yy}-F''(\xi(x, y)) w=0, \quad (x, y) \in D^\infty
\end{equation}
where $\xi(x, y)=su(x, y)+(1-s)u_{\lambda}(x, y)$ for some $s(x,y)
\in (0, 1)$  and $F''(\xi(0,0))>0$. This is a contradiction, which
leads to the claim in $D_{\lambda}^+$.

Similarly, the claim can be shown in $D_{\lambda}^-$ by the strong
maximum principle, due to the fact that $ u_{\lambda}>1-\delta $ in
$D_{\lambda}^-$ as in \eqref{asym-k}.  The claim is also true in
$D_{\lambda}^0$ when $\lambda$ is large enough,  due to the last two
estimates in \eqref{asym-k}.

Then, using the  strong maximum principle (or the Harnack
inequality) to an elliptic equation satisfied by $w_\lambda$ which
is similar to  \eqref{elliptic-w}, the lemma is proven.

\end{proof}

Now we define
 $$ \Lambda=\inf\{ \lambda :   u_\lambda(x, y) > u(x, y), (x, y)
\in D_\lambda\}.
$$

\begin{lemma} There holds
$$  \Lambda =(B_1+B_2)/2 $$
where $B_1, B_2$ are as in \eqref{limit-cons}.
\end{lemma}

\begin{proof}
We shall prove this lemma by contradiction. Suppose the lemma does
not hold.  By \eqref{limit-cons}, we can easily see that $\Lambda
>(B_1+B_2)/2$ and $w_{\Lambda}>0, \forall (x, y) \in D_{\Lambda}$.
 Then there exists a sequence of numbers $\{\lambda_m\} $
 such that $\lambda_m <
\Lambda $,  and $\lim_{m \to \infty} \lambda_m =\Lambda$ and the
infimum of $ w_{\lambda_m}$ in $D_{\lambda_m}$ is negative.  Using
\eqref{lim-u} and  the translating arguments as above, we can show
that the infimum of $ w_{\lambda_m}$ in $D_{\lambda_m}$ is achieved
at a point $(x_m, y_m)$, i.e.,
\begin{equation}\label{infimum}
w_{\lambda_m} (x_m, y_m)=\inf_{D_{\lambda_m}} w_{\lambda_m}<0.
\end{equation}
Since $w_{\lambda_m}$ satisfies an elliptic equation similar to
\eqref{elliptic-w} with $\xi(x_m, y_m)= s u(x_m, y_m)+(1-s)
u_{\lambda_m}(x_m, y_m)$ for some $s \in (0,1)$,  by the strong
maximum principle we know that $u (x_m, y_m)> -1+\delta$ and hence
   $y_m-k_1(x_m)<R_2$  if $x_m>X_1$.  By \eqref{lim-u}  and the assumption
$\Lambda> (B_1+B_2)/2$,  we know $x_m <X_2$ for some constant $X_2$
independent of $m$. Therefore there exists a subsequence of $\{m \}$
(still denoted by itself) such that $(x_m, y_m)$ converges to $(x_0,
y_0) \in D_{\Lambda}$  and $w_{\lambda_m} $ converges  to
$w_{\Lambda}$ in $C^3_{loc}(D_{\Lambda})$ as well as  in
$C^3(B_1(x_0, y_0)\cap \bar{D_{\Lambda}})$. It is easy to see that
$\nabla w_{\Lambda}(x_0, y_0)=0$.  Furthermore, $w_{\Lambda} $ is an
even function in $x$ and  satisfies an elliptic equation similar to
\eqref{elliptic-w} in $D_{\Lambda}$, by the Harnack inequality we
can see that $(x_0, y_0) $ is not on the $y$-axis. Hence $(x_0,
y_0)$ must be on the portion of boundary $\{ (x, y): y= \Lambda\} $
of $D_{\Lambda}$. Then by the Hopf Lemma, we have
$\frac{\partial}{\partial y} w_{\Lambda}(x_0, y_0)>0$. This is a
contradiction, which proves the lemma.
\end{proof}

We note that  $ u_\Lambda \ge u$ in $D_{\Lambda}$ and $u_{y}(x,
\lambda)=-\frac{1}{2} \frac{\partial }{\partial y} w_{\lambda}(x,
\lambda)<0, \forall x \in \Er$ when $\lambda> \Lambda$. Similarly,
we can use the moving plane method from below, i.e., repeating the
above procedure for $w_{\lambda}:=$ in $D^c_{\lambda}:=\{ (x, y):
x>0, \, y<\lambda\}$, and conclude $ u_\Lambda \ge u$ in
$D^-_{\Lambda}$. Therefore, Theorem 1.1 is proven.

\section{ Even symmetry of entire solutions with four ends}

We shall show that  certain entire solutions of
\eqref{allencahn-general} with {\it four ends} must be evenly
symmetric with respect to both $x$-axis and $y$-axis after a proper
translation and rotation.  First we consider the case that  the {\it
four ends}  are asymptotically straight lines, i.e., on  each $0$-
level set $\Gamma_i$  there holds
\begin{equation}\label{fourends}
 y =\tan(\theta_i) x+ A_i+o(1)  \, \, \text{as } \,\, x \to \infty,  \,\, 1 \le i \le 4
 \end{equation}
 where $ 0 < \theta_i< \theta_{i+1}< 2 \pi, $ and  $  \theta_i \not = \pi/2,  \,\, \theta_i \not =\pi/2
,\,  \, i=1, 2, 3, 4$. Without loss of generality, after a proper
rotation we may also assume that $ 0<\theta_1=2\pi-\theta_4< \pi/2$
and $ \theta_2 \not =\pi, \theta_3 \not = \pi$.

By Proposition \ref{exponential1}, we know that Hamiltonian identity
\eqref{eq:axial-ham} holds.    Moreover,   in view of \eqref{lim-u},
on a  fixed cone $ \{ (r, \theta)=(x, y):   \theta_{i-1}+\delta <
\theta < \theta_{i+1}-\delta\}$ with a sufficiently small $\delta>0$
there holds
\begin{equation}\label{asym-i}
|u(x, y)-g\bigl( x\sin\theta_i- y\cos\theta_i +A_i\cos \theta_i
\bigr)| \to 0,  \quad \text{ uniformly as } \,\,  r \to  \infty
\end{equation}

As in \eqref{limit-rho},  by Hamiltonian identity
\eqref{eq:axial-ham} we can easily obtain that
$$
\rho(x)=  \be (\cos\theta_1 +\cos\theta_4) = \be (-\cos
\theta_2-\cos\theta_3).
$$

Similarly, when $ x$-axis is replaced by $y$-axis in  Hamiltonian
identity \eqref{eq:axial-ham}, we obtain
$$
\be(\sin\theta_1+\sin\theta_2)=\be (-\sin\theta_3-\sin\theta_4).
$$
We can easily derive that
$$\pi-\theta_2= \theta_1=\theta_3-\pi.
$$

Now we follow the moving plane procedure as in the proof of Theorem
1.1. It can be shown that Lemma 3.4 still holds with $D_{\lambda}$
being modified as $\{ (x, y):  y \ge \lambda\}$.   Furthermore,
Lemma 3.5 also holds with
$$
\Lambda =\max\{  (A_1+A_4)/2,  (A_2+A_3)/2 \}.
$$
Without loss of generality, after proper translation in $y$ we may
assume that $ \Lambda =A_1+A_4=0\ge A_2+A_3$.

Next we shall show
\begin{equation}\label{2=3}
A_2+A_3=0.
\end{equation}

For this purpose, let us now state another Hamiltonian identity for
$u$, which was used in \cite{felmer} and \cite{gmx} for solutions of
nonlinear Schrodinger equation before.  A similar identity for
certain  parabolic equations is also used in \cite{polacik1} and may
be regarded as conservation of moment.

Define
\begin{equation}\label{rho-x}
E(x)= \int_{\Er} y [ F \bigl(u(x,y)\bigr) +\frac{1}{2} u_y^2(x,
y)-\frac{1}{2} u_{x}^2 (x, y)]dy.
\end{equation}
Then, by \eqref{exponential1}, $E(x)$ is well defined.  We have
\begin{proposition}
\begin{equation}
E(x) \equiv C, \quad x \in \Er.
\end{equation}
\end{proposition}

The proof of this Hamiltonian identity is based on
\eqref{exponential1} and is  similar to those in \cite{felmer} and
\cite{gmx}. The details is left to the reader.

Now, using  \eqref{asym-i},   straight forward computations can lead
to
$$
\lim_{x \to \infty} E(x) =(A_1+A_4) \be \cos\theta_1=0
$$
and
$$
\lim_{x \to -\infty} E(x) =(A_2+A_3) \be \cos\theta_1.
$$
Therefore, \eqref{2=3} is proven.

The moving plane method then leads to the even symmetry  and
monotonicity  of $u$ in $y$. Repeatting the above arguments  with
$x$ and $y$ switched,  we can show the even symmetry and
monotonicity of $u$ in $x$.  Therefore, we have shown
\begin{theorem}\label{theorem-fourends}
Assume that $u$ is an entire solution with {\it four ends}
satisfying \eqref{fourends}.  Then, after a proper translation and
rotation, $u$ satisfies \eqref{symmetry} and \eqref{monotone-u}.
\end{theorem}

\section{Energy quantization of entire solutions}

In this section we shall show that \eqref{fourends} holds under very
mild conditions on  $u$. Indeed, we shall consider entire solutions
with { \it $ 2k $ ends} in general and show some energy quantization
properties for entire solutions with {\it finite Morse index}.

\begin{lemma}\label{energy-monotone}
Suppose $u$ is an entire solution of \eqref{allencahn-general} with
{\it $2k$  ends}.  Assume
\begin{equation}\label{technical}
 \theta_{i}^+-\theta_i^- <\pi, \, \,
1\le i \le 2k.
 \end{equation}
 Then
\begin{equation}\label{energy-bound}
\cale_R(u) \le C R, \quad \forall R
\end{equation}
for some positive constant $C$.
\end{lemma}

\begin{proof}

We only need to focus on  conic region  $\calc_1$ and show

\begin{equation*}
\int_{B_R \cap \calc_1} \bigl( \frac{1}{2} |\nabla u|^2 +F(u) \bigr)
dxdy \le C R, \quad \forall R.
\end{equation*}
 Without loss of generality, we may assume
 \begin{equation}\label{cone1}
 0< \theta_1^-< \pi/2, \,\,\,  \pi/2< \theta_1^+<\pi.
 \end{equation}
  Choose  $ 0<\alpha^-<\theta_1^-, \, \theta_1^+< \alpha^+<
 \pi$ and let $\calc_1^+:= \{ (r, \theta):  \alpha^- <
\theta <\alpha^+ \}.$    Define
 $$
 \rho_1(y):= \int_{ y \cot\alpha^+}^{y\cot\alpha^-}[
 F(u)+\frac{1}{2} (u_x^2-u_y^2)] dx.
 $$
 Then, in view of \eqref{exponential1},  it is easy to see that

\begin{equation*}
\begin{split}
| \rho_1'(y)| =& |[ F(u)+\frac{1}{2} (u_x^2-u_y^2) +u_x u_y
]\mid_{x=y\cot
 \alpha^+}^{x=y\cot \alpha^-} |\\
  \le  & C e^{ -\mu_1 y}, \quad \forall y \ge R_0
 \end{split}
 \end{equation*}
 for some positive constants  $C, \mu_1$.
Hence we have
\begin{equation}\label{modify}
|\rho_1(R_1)-\rho_1(R_2) | \le Ce^{-\mu_1 R_1}, \quad \forall R_1
\le R_2
\end{equation}
for some constant $C>0$. In particular, we have
$$ |\rho_1(y)| \le C,\quad  \forall y \ge R_0.
$$
By\eqref{gradient},  we have
$$
F(u)+\frac{1}{2} (u_x^2-u_y^2) \ge  \frac{1}{2} u_x^2.
$$
Hence
\begin{equation}
\int_{B_R\cap \calc_1^+} u_x^2 dxdy  \le C R< \infty
\end{equation}
for some constant $C>0$.

Now we choose another Cartesian coordinates $(x', y')$ so that the
$x'$-axis is a small rotation of $x$-axis and \eqref{technical} and
\eqref{cone1} still hold.   Then we can obtain
\begin{equation*}
\int_{B_R\cap \calc_1^+ } u_{x'} ^2 dxdy=\int_{B_R\cap \calc_1^+}
u_{x'} ^2 dx'dy' \le C< \infty
\end{equation*}
Therefore we obtain
\begin{equation*}
\begin{split}
&\int_{B_R \cap \calc_1^+} \bigl( \frac{1}{2} |\nabla u|^2 +F(u)
\bigr) dxdy \\
& \le \int_{B_R\cap \calc_1^+}  \bigl( F(u)+\frac{1}{2} (u_x^2
-u_y^2) \bigr) dx dy +  C \int_{B_R\cap \calc_1^+}( u_x^2+u_{x'}^2 )
dxdy \\
&\le C R, \quad \forall R>0.
\end{split}
\end{equation*}
Similarly, we can show that this estimate holds for all $i\in [1,
2k]$.

In view of \eqref{exponential1}, it is easy to see  that
$$
\int_{\Er^2 \setminus \cup_{i=1}{2k}\calc_i^+} \bigl( \frac{1}{2}
|\nabla u|^2 +F(u) \bigr) dxdy \le \int_{0}^\infty C r e^{-\mu r}dr
< C
$$
for some constant $C>0$. Hence \eqref{energy-bound} is proven.

\end{proof}

In \cite{mod2},  Modica showed  Proposition \ref{modica2}  which
says that $ \cale_R(u)/R$ is increasing in $R$.  It follows
immediately that $ \lim_{R \to \infty} \cale_R(u)/R $ exists.
Indeed, we can show  the following energy quantization property for
entire solutions with {\it finite Morse index}.

\begin{lemma}\label{limit-lines}
Assume that $u$ is an entire solution of \eqref{allencahn-general}
with {\it finite Morse index} and {\it $2k$  ends}.   Assume also
the technical condition \eqref{technical}.  Then the $0$-level sets
$\Gamma$ of $u$ are asymptotically straight lines, i.e.,
 there exist $\theta_i \in [\theta_i^-, \theta_i^+], \, \, 1 \le i \le 2k$
such that on $\Gamma_i$
\begin{equation}\label{2kends}
 y =\tan(\theta_i) x+ A_i+o(1)  \, \, \, \, \text{as } \,\, x \to \infty,  \,\, 1 \le i \le
 2k
 \end{equation}
 where   $  \theta_i \not = \pi/2,  \theta_i \not = 3\pi/2,
 \, \, \forall i \in [1, 2k]$ after a proper rotation.
Moreover, \eqref{energy-limit} holds.
\end{lemma}

\begin{proof}

It is easy to see that $u_\epsilon(x):=u(x/\epsilon )$ is a critical
point of   functional

\begin{equation}\label{energy-e}
\cale_{\epsilon, R} (u)=\int_{B_R\setminus B_{1/(2R)}} \bigl(
\frac{\epsilon}{2} |\nabla u|^2 + \frac{1}{\epsilon} F(u) \bigr)
dxdy.
\end{equation}
Fix  $R=1$, $u_{\epsilon}$ is a {\it stable} critical point of
\eqref{energy-e} with $ \cale_{\epsilon, 1} (u_\epsilon) <C<
\infty$. By a $Gamma-$ convergence result of Tonegawa (Theorem 5 in
\cite{yoshi} ), there exists a sequence $\epsilon_n$
 and a union $L$ of  $N$ non-intersecting lines of $B_{1}\setminus B_{1/2} $
 such that
\begin{equation}\label{limit-e}
 \epsilon_n\cdot \bigl(\Gamma \cap  (B_{R/\epsilon_n} \setminus
B_{1/(2\epsilon_n R)})  \bigr)  \to L  \,\, \text{  in Hausdorff
distance as  } \,\, n \to \infty.
\end{equation}
Now fix $R=2, 3, \cdots$ and repeat the argument above for a
subsequence of $\{\epsilon_n\} $  in  the previous  step, by the
diagonal procedure we can find a subsequence, still denoted by
$\epsilon_n$, such that \eqref{limit-e} holds for all $R=1, 2,
\cdots$.  Therefore $L$ must be the union of  $N$ different rays
starting from the origin, and
\begin{equation}\label{quantization}
\lim_{R \to \infty} \cale_R(u)/R = N \be.
\end{equation}

 Fix a ray in $L$. Without loss of generality,  we may assume it to
be the positive $x$-axis which belongs to $\calc_1$ after some
rotation.  Then, for any fixed small angles $\alpha_2>\alpha_1>0$,
there exists  a sequence of conic regions $ \calc_{R_n, M_n,
\alpha_i}:=\{ (x, y): R_n \le x \le M_n, \, \, |y| \le \tan \alpha_i
\}, \, i=1, 2 $ such that $R_n \to \infty, M_n/R_n \to \infty$ and
$$\calc_{R_n, M_n, \alpha_2}\cap \Gamma \subset \calc_{R_n, M_n,
\alpha_1}.
$$

 On the other hand,  thanks to the stability of $u$ in $\Er^2\subset B_{R_0}$ when $R_0$
is large enough,  by similar arguments to the proof of \eqref{lim-u}
we can show that
$$\calc_{R_n, M_n, \alpha_2}\cap \Gamma =\{ (x, y): y= k(x), \,\,  R_n \le x \le M_n \}
$$
for some $C^2$ function $k(x)$  and
\begin{equation}\label{condition}
\max_{ x \in [R_n, M_n]} |k'(x)|< \tan \alpha_1, \quad \max_{ x \in
[R_n, M_n] } |k''(x)| \to 0, \quad \text{as }  n \to \infty .
\end{equation}
Moreover,
$$
||u(x, y)-g(y-k(x))||_{C^2(\calc_{R_n, M_n, \alpha_2} )} \to 0,\quad
\text{ as } n \to \infty.
$$

 We may also assume that $k'( R_n) \to 0$.
We claim that when $n$ is large enough, $M_n$ can be chosen as any
number $R>R_n$ and \eqref{condition} still holds.   If this is not
true, we can choose $M_n$ such that \eqref{condition} holds  but
$k'(M_n)=\tan\alpha_1.$  We claim that $\bigl(\calc_{R_n, M_n,
\alpha_2} \setminus  \calc_{R_n, 2M_n, (\alpha_1+\alpha_2)/2} \bigr)
\cap B_{M_n} (M_n, k(M_n)) $  is empty. If we assume otherwise,
without loss of generality, we may assume that $M_n$ is the first
such sequence related to a ray in $L$.  Now  we use
$\epsilon_n=1/M_n$ as in \eqref{limit-e}, and obtain  the limit as
$L'$ which  is the union of at least $N+1$ rays.  This is  a
contradiction to \eqref{quantization}.  Hence  the claim is true.
Then, using the modified Hamiltonian identity in $\calc_{R_n,
M_n,\alpha_2}$ as in \eqref{modify} with the $y$-axis being replaced
by the tangential direction  of $k(x)$ at $(M_n, k(M_n)$, we obtain
$$
\be  \le \be \cos \alpha_1+o(1), \quad \text{ as } n \to \infty.
$$
This is a contradiction, and hence proves that $M_n$ can be chosen
as any $R>R_n$ when $n$ large enough.   Therefore
$$\calc_{R_n, \infty, \alpha_2}\cap \Gamma =\{ (x, y): y= k(x), \, \, x > R_n \}
$$
and
$$
|k'(x)| < \tan\alpha_1, \quad x>R_n.
$$
Since $\alpha_1>0$  is arbitrary, we obtain that
$$
\lim_{x \to \infty} k'(x) =0.
$$
Now use Lemma \ref{straightline}, we conclude that $\Gamma \cap
\calc_1$ is asymptotically straight line.   The lemma then follows.

\end{proof}

\begin{remark}
Given that $u$ satisfies the condition in Theorem 1.2. If we assume
further that, after a proper rotation, the level set in $\calc_i$
outside a large ball $B_R$  is a graph of a $C^2$ function $k(x)$,
i.e.,
\begin{equation}\label{graph}
\Gamma \cap \calc_i \cap B^c_{R}=\{ (x, y): y= k(x), \, \, x > R \},
\quad  1\le i \le 2k,
\end{equation}
then the conclusion of  Lemma \ref{limit-lines} can be shown
directly without using the result in \cite{yoshi}.  We just start
the proof from \eqref{condition} with $M_n = \infty$   and exploits
the modified Hamiltonian identity. The details is omitted.
\end{remark}

Theorem 1.2 follows from Lemma \ref{limit-lines} and Theorem
\ref{theorem-fourends} directly. If we replace \eqref{tech-4} in
Theorem 1.2 by \eqref{energy-bound}, the conclusion of Theorem 1.2
still holds.

\noindent {\bf Proof of Theorem 1.3}

If \eqref{energy-infinity} does not hold, by the monotonicity
formula of Modica we know that \eqref{energy-bound} must be true.
Using  the $\Gamma-$ convergence result of Tonegawa as in the proof
of Lemma \ref{limit-lines}, we know that there exists a sequence
$\{R_n\}$ such that $R_n \to \infty$ and
\begin{equation}\label{limit-gamma}
 \frac{1}{R_n} \cdot \bigl(\Gamma \cap  B_{M R_n}   \bigr)  \to L  \,\, \text{  in Hausdorff
distance as  } \,\, n \to \infty
\end{equation}
for any $M>0$, where $L$ is the union of $N$  rays from the origin.
Moreover,  \eqref{quantization} holds.  It follows that  $\Gamma$
must be asymptotically straight lines at infinity, as in the proof
of Lemma \eqref{limit-lines}.    Note that  $\Gamma$ is a  union of
$C^2$ curves except at singular points where $u $ and $\nabla u$
both vanish, and $u$  $u$ behaves like harmonic function  near these
singular points.  Therefore $N$ must be an even positive integer
$2k$.  We denote  the directions of these lines by
 $\nu_i= (\cos \theta_i, \sin \theta_i), 1\le i \le 2k$ with
$0< \theta_i< \theta_{i+1} < 2\pi,  1\le i \le 2k-1 $, after a
proper rotation.  Using Hamiltonian identity similar to
\eqref{ham-const} but with more terms (see also \cite{gui1} ), we
obtain
\begin{equation}\label{ham-2k}
\sum_{i=1}^{2k} \be \sin (\theta_i+\theta)=0
\end{equation}
for  almost all $\theta$. Hence \eqref{balance} holds.  The proof of
Theorem 1.3 is complete.

{\bf Acknowledgement} This  research  is partially supported by
National Science Foundation Grant DMS 0500871.

\end{document}